\documentclass[12pt]{amsart}

\usepackage{amsmath,
amsthm,amsfonts,amssymb, latexsym,enumerate,url}

\usepackage[colorlinks=true, 
linkcolor = blue,
citecolor=black,
urlcolor=black,
]{hyperref}

\usepackage{microtype}

\usepackage[backend=biber,
sorting=nyt,
style=numeric,
maxnames=99,
minnames=99,
]{biblatex}
\DeclareFieldFormat{prefixnumber}{\mkbibbold{#1}}
\DeclareFieldFormat{labelnumber}{\mkbibbold{#1}}
\DeclareFieldFormat[article]{number}{\mkbibparens{#1}}
\DeclareFieldFormat{addendum}{\mkbibparens{#1}}

\renewbibmacro*{note+pages}{%
  \printfield{note}
  \setunit{\addspace}
  \printfield{pages}
}

\DeclareFieldFormat{pages}{#1} 
\DeclareFieldFormat{vol}{#1}
\renewbibmacro{in:}{}

\DeclareFieldFormat[article]{title}{\textnormal{#1}}  

\DeclareFieldFormat[article]{volume}{\textbf{#1}}

\makeatletter
\renewenvironment{proof}[1][\proofname]{\par\pushQED{\qed}%
	\normalfont \topsep6\p@\@plus6\p@\relax
	\trivlist
	\item\relax
		{\bfseries
	#1\@addpunct{.}}\hspace\labelsep\ignorespaces}{%
	\popQED\endtrivlist\@endpefalse
}
\makeatother

\renewcommand{\bf}{\textbf}

\newcommand{\te}[1]{\text{#1}}

\newcommand{\A}{\mathsf{A}}
\newcommand{\SY}{\mathsf{S}}
\newcommand{\CY}{\mathsf{C}}

\renewcommand{\bar}{\overline}

\newcommand{\ls}{<}
\newcommand{\g}{>}

\newcommand{\sm}{\setminus}
\newcommand{\Span}[1]{\langle #1\rangle}

\newcommand{\nr}[1][G]{\textnormal{\bf{N}}_{#1}}
\newcommand{\cn}[1][G]{\textnormal{\bf{C}}_{#1}}
\newcommand{\op}{\textnormal{\bf{O}}}
\newcommand{\zn}{\textnormal{\bf{Z}}}

\newcommand{\IRR}{\textnormal{Irr}}

\newcommand{\GL}{\textnormal{GL}}

\newcommand{\CD}{\textnormal{cd}}

\newcommand{\PSL}{\textnormal{PSL}}

\newcommand{\Mod}[1]{\ \mathrm{mod}\ #1}

\newcommand{\syl}[1][p]{\textnormal{Syl}_{#1}}

\renewcommand{\subset}{\subseteq}

\begin{document}

\theoremstyle{plain}

\newtheorem{thm}{Theorem}[section]
\newtheorem{lem}[thm]{Lemma}
\newtheorem{conj}[thm]{Conjecture}
\newtheorem{pro}[thm]{Proposition}
\newtheorem{cor}[thm]{Corollary}
\newtheorem{que}[thm]{Question}
\newtheorem{rem}[thm]{Remark}
\newtheorem{defi}[thm]{Definition}
\newtheorem{hyp}[thm]{Hypothesis}

\newtheorem*{thmA}{THEOREM A}
\newtheorem*{thmB}{THEOREM B}
\newtheorem*{corC}{COROLLARY C}

\newtheorem*{thmC}{THEOREM C}
\newtheorem*{conjA}{CONJECTURE A}
\newtheorem*{conjB}{CONJECTURE B}
\newtheorem*{conjC}{CONJECTURE C}

\newtheorem*{thmAcl}{Main Theorem$^{*}$}
\newtheorem*{thmBcl}{Theorem B$^{*}$}

\numberwithin{equation}{section}

\marginparsep-0.5cm

\renewcommand{\thefootnote}{\fnsymbol{footnote}}
\footnotesep6.5pt

\title{Minimal-order groups with an irreducible character of degree \(p\) or \(p^2\)}

\author{Asier Arranz }
\email{aarranz97@alumno.uned.es}


\begin{abstract}
We characterize the finite groups of minimal order that admit an irreducible complex character of degree \(p\) or \(p^2\), where \(p\) is a prime. 
\end{abstract}

\thanks{This research was conducted as part of my undergraduate thesis at UNED directed by G. Navarro, whom I thank for their guidance and useful comments.}

\maketitle

\section{Introduction}

G. Navarro has proposed us the following question:
Given an integer \(n\g 1\), what is the smallest number 
\(m=g(n)\) such that there exists a finite group of order \(m\)
with an irreducible complex character of degree \(n\)? In fact, some results about this question have appeared in \cite{HI} (see also \cite{SE}). 

In full generality, this problem appears too difficult to admit a complete answer.

Suppose for instance that \(n=p\) is a prime. Then an extraspecial group of order \(p^3\) has an irreducible character of degree \(p\). (Taking direct products, we see that any number $n$ is the irreducible character degree of some finite group.) Using Dirichlet’s theorem on arithmetic progressions, let \(q\) be the smallest prime such that \(q\equiv 1\Mod{p}\). Then the Frobenius group of order \(pq\) has an irreducible character of degree \(p\). An open conjecture by H. Kanold \cite{K64}, asserts that \(q\ls p^2\). Hence, it seems that we have to choose \(pq\) over \(p^3\).  But there is more: the projective special linear group \(\PSL_2(p)\) has an irreducible
character of degree \(p\) and it is of order \(|\PSL_2(p)| = p(p^2-1)/(2,p-1)\). However, it is not known when \(q<(p^2-1)/(2,p-1)\). 

By using \(\mathsf{GAP}\) \cite{GAP} or elementary methods, we quickly check the first few cases.
For instance, \(g(2)=6\) with \(G=\SY_3\), \(g(3)=12\) with \(G=\A_4\), \(g(4)=20\) with \(G=\CY_5\rtimes \CY_4\), \(g(5)=55\) with \(G=\CY_{11}\rtimes \CY_5\), \(g(6)=42\) with \(G=\CY_7\rtimes \CY_6\), \(g(7)=56\) with \((\CY_2)^3\rtimes \CY_7\), \(g(8)=72\) with either \(G=(\CY_3)^2 \rtimes D_8\) or \(G=(\CY_3)^2 \rtimes Q_8\), and \(g(9)=144\) with \(G=\A_4\times \A_4\).  

In this note, we partially answer Navarro's question by characterizing the minimal-order groups with an irreducible character of degree $n$, for the first cases where $n$ is a prime or the square of a prime. Namely, we prove the following.

\begin{thmA} \label{thmA:theoA}
Let $G$ be a finite group of minimal order with an irreducible character of prime degree $p$. Then exactly one of the following occurs.
\begin{enumerate}[(a)]
\item \(G=\PSL_2(p)\).
\item \(G\) is a Frobenius group of the form \((\CY_q)^m\rtimes \CY_p\), where \(q^m\) is the least prime-power such that \(p\) divides \(q^m-1\). 
\end{enumerate}
\end{thmA}

A computational check shows that, for primes \(p\ls 4000\), case (a) of Theorem~\hyperref[thmA:theoA]{A} occurs only for \(p=19\). We are not aware of any more examples where case (a) holds.

\begin{thmB} \label{thmB:theoB}
Let $G$ be a finite group of minimal order with an irreducible chararacter of degree \(p^2\), where $p$ is a prime. Then exactly one of the following occurs.
\begin{enumerate}[(a)]
\item $G$ is any \(p\te{-group}\) of order \(p^5\) with an irreducible character of degree \(p^2\). 
\item \(G\) is a Frobenius group of the form \((\CY_q)^m\rtimes \CY_{p^2}\), where \(q^m\) is the least prime-power such that \(p^2\) divides \(q^m-1\). 
\item $G$ is a direct product of two Frobenius groups of the form \((\CY_q)^m\rtimes \CY_p\), where \(q^m\) is the least prime-power such that \(p\) divides \(q^m-1\). 
\end{enumerate}
\end{thmB}

Somewhat similarly, for primes \(p\leqslant 71\) case (a) of Theorem~\hyperref[thmB:theoB]{B} occurs only for \(p=19\). Again, we are not aware of any more examples where case (a) holds. 
\medskip

Of course, if \(G\) is a group of minimal order with an irreducible character \(\chi\in \IRR(G)\) of degree \(\chi(1)=n\), then \(\chi\) is faithful by the minimality of $|G|$. Also, the restriction \(\chi_H\) of \(\chi\) to a proper subgroup \(H\ls G\) is not irreducible. Thus $G$ is minimally irreducible in the sense of \cite{V03}. These groups have also been studied in \cite{S73, K781, K782, K79}, although we shall not use any of their results.

\section{Preliminaries} \label{prel}

We begin with a few elementary observations that will be used frequently, often without explicit reference. The first one is a consequence of Problem 2.8 of \cite{CTFG}

\begin{lem}\label{lem:faithabe}
Let $G$ be a group of minimal order with an irreducible character \(\chi\in \IRR(G)\) of degree \(\chi(1)=n\) and let \(H\subset G\). If every irreducible constituent of \(\chi_H\) is linear, then \(H\) is abelian.
\end{lem}

\begin{proof}
Let \(x,y\in H\). Since every irreducible constituent is linear, we have \(\chi([x,y])=\chi(1)\) and so \([x,y]\in \ker \chi = 1\). Therefore, $H$ is abelian. 
\end{proof}

The following is a consequence of Corollary 11.19 of \cite{CTFG}. 

\begin{lem}\label{lem:divp}
Let $G$ be a group of minimal order with an irreducible character \(\chi\in \IRR(G)\) of degree \(\chi(1)=n\). If \(N\lhd G\) is a proper normal subgroup of $G$, then \((|G:N|,n)\g 1\). 
\end{lem}

\begin{proof}
Let \(N\lhd G\) be normal and suppose that \(|G:N|\) and \(\chi(1)=n\) are coprime. Then \(\chi_N\in \IRR(N)\) by Corollary 11.19 of \cite{CTFG} and \(N=G\) by the minimality of \(|G|\). 
\end{proof}

We also recall that if \(n\) is an irreducible character degree of a group $G$, then \(n\) divides \(|G|\) and \(n^2\ls \sum_{\chi\in \IRR(G)} \chi(1)^2 = |G|\).

\section{Proof of Theorem A}

As we stated in the introduction, the smallest groups with an irreducible character of degree \(p=2,3\) are \(\SY_3\) and \(\A_4\), respectively, so Theorem~\hyperref[thmA:theoA]{A} holds in this case. For \(p\g 3\), the smallest simple group with an irreducible character of degree \(p\) is \(\PSL_2(p)\) by the main result of \cite{B58}.  It then remains to prove the nonsimple case of Theorem~\hyperref[thmA:theoA]{A}. 

\begin{thm}
Let $G$ be a group of minimal order with an irreducible character \(\chi\in \IRR(G)\) of prime degree \(\chi(1)=p\g 3\). If \(G\) is nonsimple, then Theorem~\hyperref[thmA:theoA]{A} holds in this case. 
\end{thm}

\begin{proof}
First, note that \(|G|\leqslant |\PSL_2(p)|=p(p^2-1)/2\ls p^3\) by the minimality of \(|G|\). Let \(M\lhd G\) be maximal normal so that \(G/M\) is simple. By Lemma~\ref{lem:divp}, \(p\) divides \(|G:M|\ls |G|\leqslant |\PSL_2(p)|\) and, by the main result of \cite{B58}, \(\PSL_2(p)\) is the smallest nonabelian simple group of order divisible by a prime \(p\g 3\).  Thus \(G/M\) is abelian simple and so \(|G:M|=p\). Then \(\chi_M\) is a sum of linear characters by Corollary 6.19 of \cite{CTFG} and $M$ is abelian by Lemma~\ref{lem:faithabe}.

If a Sylow \(p\te{-subgroup}\) \(P\in \syl(G)\) is normal in $G$, then \(G=P\) by Lemma~\ref{lem:divp} and \(G\) is a \(p\te{-group}\). Hence, \(|G|\geqslant p^3\g |\PSL_2(p)|\) against the minimality of \(|G|\). We deduce that \(P\) cannot be normal and, by Sylow theory, \(|G:\nr(P)|-1\geqslant p\). Thus \(|P|=p\) and \((|M|,|P|)=1\). Letting \(P\) act on $M$ by conjugation, we can write \(M=\cn[M](P)\times [M,P]\) by Fitting's theorem \cite[Theorem~4.34]{FGT} and since \(G\) is nonabelian, \([M,P]\g 1\) is nontrivial. Let \(N\subset [M,P]\lhd MP=G\) be minimal normal in $G$ so that $N$ is elementary abelian and \(1=\cn[N](P)=\cn[N](x)\) for every \(1\neq x\in P\). Then \(NP\) is a Frobenius group with Frobenius complement \(P\) and, by Theorem 6.34 of \cite{CTFG}, \(NP\) has an irreducible character of degree \(|P|=p\). The minimality of \(|G|\) yields \(G=NP=(\CY_q)^m\rtimes \CY_p\) for some prime \(q\) and an integer \(m\geqslant 1\), as required. 

Now, if \(q^m\) is the least prime-power such that $p$ divides \(q^m-1\), then as \(p\) divides the order of \(\GL_m(q)\) we may choose \(x\in \GL_m(q)\) of order $p$. Let \(P=\Span{x}\) act on a vector space $N$ of dimension $m$ over the field of $q$ elements. Again by Fitting's theorem, \(N=\cn[N](P)\times [N,P]\) and the action of \(P\) on \([N,P]\) is Frobenius. By elementary properties of Frobenius groups, \(|P|=p\) divides \(|[N,P]|-1\) and the minimality of \(q^m\) forces \(\cn[N](P)=1\) and \([N,P]=N\).  It follows that \(G=N\rtimes P\) is a Frobenius group and, as before, $G$ has an irreducible character of degree $p$. This completes the proof. 
\end{proof}

\section{Proof of Theorem B}

Let $G$ be a group of minimal order with an irreducible character of degree \(p^2\). Since there exist \(p\te{-groups}\) of order \(p^5\) with an irreducible character of degree \(p^2\) (see \cite{P5}), we may assume from now on that \(|G|\ls p^5\).

To prove Theorem~\hyperref[thmB:theoB]{B}, we begin by assuming that $G$ is solvable. In fact, it suffices to impose the following weaker condition. 

\begin{thm}\label{thm:solvableB}
Let $G$ be a group of minimal order with an irreducible character \(\chi\in \IRR(G)\) of degree \(\chi(1)=p^2\), where \(p\) is a prime. If \(G''\ls G'\ls G\), then $G$ is solvable and Theorem~\hyperref[thmB:theoB]{B} holds in this case. 
\end{thm}

Our proof Theorem~\ref{thm:solvableB} relies on Lemma 12.3 of \cite{CTFG}, part of which we state here for the reader's convenience.

\begin{lem}\label{lem:12.3}
Let $G$ be solvable and assume that \(G'\) is the unique minimal normal subgroup of $G$. Then exactly one of the following occurs.
\begin{enumerate}[(a)]
\item $G$ is an extraspecial \(p\te{-group}\).  
\item \(G\) is a Frobenius group with elementary abelian Frobenius kernel \(G'\) and abelian Frobenius complement. 
\end{enumerate}
\end{lem}

The way  Lemma~\ref{lem:12.3} is applied to the proof of Theorem~\ref{thm:solvableB} is the following. Let \(G''\subset K\lhd G\) be maximal such that \(G/K\) is nonabelian. Then  \(G/K\) is solvable and \((G/K)'\) is the unique minimal normal subgroup of \(G/K\). We may thus apply Lemma~\ref{lem:12.3} to the group \(G/K\). 

\medskip

In order to handle case (c) of Theorem~\hyperref[thmB:theoB]{B}, we use the following lemma, whose proof we present after proving Theorem~\ref{thm:solvableB}.

\begin{lem}\label{lem:Bc}
Let \(G\) be a group of minimal order with an irreducible character \(\chi\in \IRR(G)\) of degree \(\chi(1)=p^2\). If there exists \(K\lhd G\) such that \(G/K\) and \(K\) are Frobenius groups of the form \((\CY_q)^m\rtimes \CY_p\), then \(G\cong K\times (G/K)\). 
\end{lem}

We will also use the fact that a Singer subgroup of \(\GL_m(q)\) is cyclic of order \(q^m-1\) (see \cite[Satz~II.7.3]{H67}) and, in particular, \(\GL_m(q)\) contains an element of order $d$ for every divisor $d$ of \(q^m-1\).

\begin{proof}[Proof of Theorem~\ref{thm:solvableB}]
We assume that \(|G|\ls p^5\) and we let \(G''\subset K\lhd G\) be maximal with the property that \(G/K\) is nonabelian. Then we are on the hypotheses of Lemma~\ref{lem:12.3} and we suppose first that case (b) holds so that \(G/K\) is a Frobenius group with abelian Frobenius complement and elementary abelian Frobenius kernel \(N/K=(G/K)'\). Thus \(G/N\) is abelian and it must be a \(p\te{-group}\) by Lemma~\ref{lem:divp}. 

If \(K=1\), then \(G'' = 1\) and \((|N|,|G:N|)=1\), so a Sylow \(p\te{-subgroup}\) \(P\in \syl(G)\) is a Frobenius complement of $G$. Since abelian Frobenius complements are cyclic, we can argue as in the proof of Theorem~\hyperref[thmA:theoA]{A} to deduce that $G$ is of the form \((\CY_q)^m \rtimes \CY_{p^2}\). Also, if \(q^m\) is the least prime-power such that \(p^2\) divides \(q^m-1\), then \(\GL_m(q)\) contains an element of order \(p^2\). Arguing again as in the proof of Theorem~\hyperref[thmA:theoA]{A}, we can construct a Frobenius group of the form \((\CY_q)^m\rtimes \CY_{p^2}\). Thus Theorem~\hyperref[thmB:theoB]{B} holds in this case. 

If \(K\g 1\), then since \(G/K\) has an irreducible character of degree \(|G:N|\) by Theorem 6.34 of \cite{CTFG}, we have \(|G:N|=p\) by the minimality of \(|G|\). Thus \(|G:K|\geqslant |H|\) where $H$ is of minimal order having an irreducible character of degree \(p\). Let \(\theta\in \IRR(N)\) lie under \(\chi\) so that \(\theta(1)=p\) by \cite[Corollary~6.19]{CTFG}. Since \((|N:K|,p)=(|N:K|,|G:N|)=1\), we have \(\theta_K\in \IRR(K)\) and therefore \(|K|\geqslant |H|\). But \(H\times H\) has an irreducible character of degree \(p^2\), so \(|H|^2\geqslant |G|=|G:K| |K|\geqslant |H|^2\) and equality holds throughout. In particular, \(K\) and \(G/K\) are minimal-order groups with an irreducible character of degree \(p\). Since \(|K|=|H|=|G:K|\) and \(|G:K|\) is divisible by exactly two primes by Theorem~\hyperref[thmA:theoA]{A}, then \(K\) is nonsimple by Burnside's \(p^aq^b\te{-theorem}\). Again by Theorem~\hyperref[thmA:theoA]{A}, \(K\) is a Frobenius group of the form \((\CY_q)^m\rtimes \CY_p\)   . It readily follows by Lemma~\ref{lem:Bc} that Theorem~\hyperref[thmB:theoB]{B} holds in this case. 
												
Otherwise, \(G/K\) is an extraspecial \(q\te{-group}\) for a prime \(q\) and so \(|G:K|\geqslant q^3\). By Lemma~\ref{lem:divp}, \(q=p\).  Since \(p^4=(p^2)^2\ls |G|\) and we are assuming that \(|G|\ls p^5\), then $G$ is not a \(p\te{-group}\) and a Sylow \(p\te{-subgroup}\)  \(P\in \syl(G)\) cannot be normal by Lemma~\ref{lem:divp}. Therefore, \(|G:\nr(P)|-1\geqslant p\) by Sylow theory and hence \(|P|=|G:K|=p^3\) and \(|K|\ls p^2\).  

Now, if \(\lambda\in \IRR(K)\) lies under \(\chi\) then \(\lambda(1)\) divides both \(|K|\) and \(\chi(1)=p^2\), so \(\lambda(1)=1\) and \(K\) is abelian by Lemma~\ref{lem:faithabe}. Since \(K\not\subset \zn(G)\) because \(P\) is not normal and \(\zn(G)=\zn(\chi)\) by Lemma 2.27 of \cite{CTFG}, then \(\chi_K\) has at least \(p\) irreducible constituents. But \(|\IRR(K)|=|K|\ls p^2\), so Clifford's theorem yields \(\chi_K = p\sum_{i=1}^p \lambda_i\) for distinct \(\lambda_i\in \IRR(K)\) and \(i=1,\ldots,p\). Thus
\[ [\chi_K,\chi_K] =p^2\sum_{i,j=1}^p [\lambda_i,\lambda_j] =p^3 =  |G:K|[\chi,\chi] \]
and, by Lemma 2.29 of \cite{CTFG}, \(\chi\) vanishes on \(G\sm K\). In particular, \(\chi_P\) vanishes on \(P\sm \{1\}\)  and, applying Problem 2.16 of \cite{CTFG}, \(|P|=p^3\) divides \(\chi(1)=p^2\). This contradiction completes the proof of Theorem~\ref{thm:solvableB}. 
\end{proof}

We now prove Lemma~\ref{lem:Bc}. 

\begin{proof}[Proof of Lemma~\ref{lem:Bc}]
We first check that \(G\) has an abelian normal \(p\te{-complement}\). Using the bar convention, write \(\bar{G}=G/\op_{p'}(G)\) and note that \(\op_{p'}(\bar{G})=1\). Since \(P\in \syl(G)\) is abelian as a \(p\te{-group}\) of order \(p^2\), we have \[\op_p(\bar{G})\subset \bar{P}\subset \cn[\bar{G}](\op_p(\bar{G}))\subset \op_p(\bar{G}),\]
where the last containment follows by the Hall-Higman Lemma 1.2.3. \cite[Theorem~3.21]{FGT}. Thus \(\bar{P}=\op_p(\bar{G})\lhd \bar{G}\) and \(\bar{P}=\bar{G}\) by Lemma~\ref{lem:divp}. That is, \(\op_{p'}(G)\) is a normal \(p\te{-complement}\). 

To prove that \(\op_{p'}(G)\) is abelian, let \(\theta\in \IRR(\op_{p'}(G))\) lie under \(\chi\). The number \(\theta(1)\) divides both \(|\op_{p'}(G)|\) and \(\chi(1)=p^2\), so \(\theta(1)=1\) and \(\op_{p'}(G)\) is abelian by Lemma~\ref{lem:faithabe}. Applying Ito's theorem \cite[Theorem~6.15]{CTFG}, we deduce that \(\CD(G)\subset \{1,p,p^2\}\) where \(\CD(G)\) denotes the set of irreducible character degrees \(\{\psi(1)\,|\, \psi\in \IRR(G)\}\).

Write \(K=NQ\) where \(N=\op_{p'}(K)\lhd G\) is the Frobenius kernel of $K$ and \(Q\subset     K\) is the Frobenius complement of order $p$. We claim that \(Q\) centralizes some complement \(M\subset\op_{p'}(G)\) for \(N\) in \(\op_{p'}(G)\). By elementary properties of Frobenius groups, \(1=\nr[N](Q)=N\cap \nr(Q)\) and, writing \(H=\nr(Q)\), we have \(G=KH=(NQ)H=NH\) by the Frattini argument. Then \(\CD(H)=\CD(G/N)\subset \CD(G)\subset \{1,p,p^2\}\) and, since \(G/N\) is nonabelian, the minimality of \(|G|\) yields \(\CD(H)=\{1,p\}\). Theorem 12.11 of \cite{CTFG} applies and we may take \(A\lhd H\) abelian with \(|H:A|=p\). Let \(M=\op_{p'}(A)=\op_{p'}(H)\) so that it is a complement for \(N\) in \(\op_{p'}(G)\).  Both \(M\) and \(Q\) are normal in \(H=\nr(Q)\), so \([M,Q]\subset M\cap Q=1\) as claimed.

Since \(G=NH\) and \([N,M]\subset [\op_{p'}(G),\op_{p'}(G)] = 1\), then \(M\) is normal in $G$.  
Writing \(\bar{G}=G/M\), then \(K\cong \bar{K}\subset \bar{G}\), so \(\bar{G}\) is nonabelian and we may take \(M\subset L\lhd G\) maximal such that \(G/L\) is nonabelian. Arguing as in the proof of Theorem~\ref{thm:solvableB}, we deduce that \(L\) and \(G/L\) are Frobenius groups of the form \((\CY_q)^m\rtimes \CY_p\). Thus \(L=MS\) for some \(S\subset G\) of order $p$ and, since \(A=MQ\) is abelian, \(L\cap Q=1\) and  \(QS\in \syl(G)\). 

As before, we can argue that \(S\) centralizes some complement \(N_0\subset \op_{p'}(G)\) for \(M\) in \(\op_{p'}(G)\). But writing again \(\bar{G}=G/M\), we then have \(\bar{N_0} =\op_{p'}(\bar{G})=\bar{N}\), so \(1=[\bar{S},\bar{N_0}] = [\bar{S},\bar{N}]=\bar{[S,N]}\) and hence \([S,N]\subset M\). Since we also have \([S,N]\subset N\lhd G\), then \([S,N]\subset N\cap M=1\) and it follows at once that \(G=\op_{p'}(G)(QS)=(NQ)\times (MS)\cong K\times (G/K)\). 
\end{proof}

We complete the proof of Theorem~\hyperref[thmB:theoB]{B} by showing that no nonsolvable group can be a group of minimal order with an irreducible character of degree \(p^2\). As we checked in the introduction, the smallest groups with an irreducible character of degree \(p^2 = 4,9\) are \(\CY_5\rtimes \CY_4\) and \(\A_4\times \A_4\), respectively. Thus Theorem~\hyperref[thmB:theoB]{B} holds for \(p=2\) and \(p=3\).

For primes \(p\g 3\), we use the classification of finite simple groups. We shall also use the fact that a nonabelian simple group $S$ has an irreducible character \(\chi\in\IRR(S)\) of \(p\te{-defect}\) zero (that is, $p$ does not divide \(|S|/\chi(1)\)). This was stablished by G. Michler for finite simple groups of Lie type  \cite{M86} and by  Granville-Ono for alternating groups \cite{G96}. For  sporadic groups, the assertion follows from direct inspection of the ATLAS \cite{C85}.

\medskip

We begin with the following lemma.

\begin{lem}
Let $G$ be a group of minimal order with an irreducible character \(\chi\in \IRR(G)\) of degree \(\chi(1)=p^2\g 9\). If $G$ is nonsolvable, then \(G\) is semisimple. 
\end{lem}

\begin{proof}
We first note that either \(G''=G'\ls G\) or \(G'=G\) by Theorem~\ref{thm:solvableB}. Let \(N\lhd G\) be maximal normal such that \(N\subset G'\). Then \(G'/N\) is a direct product of \(r\geqslant 1\) groups isomorphic to a nonabelian simple group $S$ and \(G/N\) acts transitively on these $r$ factors. Since \(G/G'\) is abelian, it must be a \(p\te{-group}\) by Lemma~\ref{lem:divp} 
and, if \(r\g 1\), then \(r\geqslant p\). In this case,  \(|S|\leqslant |G'/N|^{1/p}\). But we are assuming that \(|G|\ls p^5\), so \(|S|\ls p^{5/p}\ls 7\) which is, of course, not true. Thus \(r=1\) and \(G'/N=S\).  

Assume that \(G''=G'\ls G\) and let \(G/N\) act by conjugation on  the nonabelian simple group \(S=G'/N\) to obtain a homomorphism \(G/N\rightarrow \te{Aut}(S)\). Let \(K/N\) be the kernel of the homomorphism and note that \(G'\cap K=N\) and \(|G:G'K|\) divides \(|\te{Out}(S)|\). If \((|S|,p)=1\) and \(\theta\in \IRR(G')\) lies under \(\chi\), then \(\theta_N\in \IRR(N)\) and, as \(G'\) is nonabelian, \(\theta(1)=p\). Hence, \(|N|\g p^2\) and as \(|G|\ls p^5\) and \(|G:G'|\geqslant p\), we have \(|S|\ls p^2\). But a Sylow \(p\te{-subgroup}\) of \(G/N\) cannot be normal by Lemma~\ref{lem:divp}, so \(p\) divides \(|G:G'K|\) and \(p\leqslant |\te{Out}(S)|\). Applying the main result of \cite{K03}, we obtain that \(p\leqslant |\te{Out}(S)|\leqslant \log_2 |S|\ls 2\log_2  p\) so that   \(p\leqslant 3\) against our assumption.

If, otherwise, \(p\) divides \(|S|\), suppose first that $p$ divides \(|G:G'K|\leqslant |\te{Out}(S)|\). Then \(|S|\ls p^4\) and applying again \cite{K03} we have \(p\leqslant |\te{Out}(S)|\leqslant \log_2 |S| = 4\log_2 p\), so \(p\in \{5,7,11,13\}\). Using \(\mathsf{GAP}\), we obtain that either \(S\in \{\A_5,\A_6\}\), \(S=M_{11}\), \(S=\PSL_3(3)\) or \(S=\PSL_2(q)\) for some \(q\in \{7,8,11,13,23,25,27\}\). But then it is known that \(|\te{Out}(S)|\in \{1,2,4\}\), contradicting that \(3\ls p\leqslant |\te{Out}(S)|\). 

We deduce that \(G=G'K\) and, since \(S\) has a character of \(p\te{-defect}\) zero and \(|S|\ls p^4\), we have \(|S|_p= p\) (where \(n_p\) denotes largest \(p\te{-power}\) dividing a positive integer $n$).  In particular, \(\chi(1)=p^2\) does not divide \(|S|=|G:K|\) and $K$ must be nonabelian by Ito's theorem \cite[Theorem~6.15]{CTFG}. Then, by Lemma~\ref{lem:faithabe} and the minimality of \(|G|\), an irreducible constituent of \(\chi_K\) must be of degree \(p\). If $H$ is of minimal order with an irreducible character of degree \(p\), we then have
\[|H|^2 \geqslant |G| =|S||K|\geqslant |\PSL_2(p)| |H|\geqslant |H|^2\]
where the second inequality holds because \(\PSL_2(p)\) is the smallest nonabelian simple group of order divisible by \(p\). Equality thus holds throughout and \(|H|=|\PSL_2(p)|\). But then \(|G|=|\PSL_2(p)|^2 \g p^5\) for \(p\g 3\), and this is a contradiction. 

It follows that \(G'=G\) and \(S=G/N\) is a nonabelian simple group. Since \(p\) divides \(|S|\) by Lemma~\ref{lem:divp}, if $N$ is nonabelian we deduce as in the previous paragraph that \(|G|=|\PSL_2(p)|^2\g p^5\), which is false. Thus \(N\) is abelian and \(\chi(1)=p^2\) divides \(|S|\) by Ito's theorem. But \(S\) has an irreducible character of \(p\te{-defect}\) zero, so \(|S|\g p^4\) and \(|N|\ls p\). Since then \(|\IRR(N)|=|N|\ls p\), we can write  \(\chi_N = p^2\lambda\) for some \(\lambda\in \IRR(N)\) and, applying Lemma 2.27 of \cite{CTFG}, \(N\subset \zn(\chi)=\zn(G)\). Thus \(N=\zn(G)\) and \(G/\zn(G)=S\) is nonabelian simple, as required. 
\end{proof}

It now suffices to prove that for a quasisimple group $G$ with an irreducible character of degree \(p^2\), the inequality \(|G|\geqslant p^5\) holds. To this end, we apply the main result of \cite{M01}, which classifies the prime-power degree irreducible characters of quasisimple groups, excluding those with \(G/\zn(G)= \A_n\) for \(n\g 18\).

\begin{thm}
Let $G$ be a group of minimal order with an irreducible character of degree \(p^2\). Then $G$ is solvable.    
\end{thm}

\begin{proof}
We assume that $G$ is semisimple and we suppose first that \(\bar{G}=G/\zn(G)\) is an alternating group of order \(n!/2\). Since \(|\zn(G)|\ls p\) by the proof of the previous lemma,  \(p^2\) divides \(|\bar{G}|=n!/2\) and we have \(n\geqslant 2p\). But \(p^5\ls (2p)!/(2p-6)!\leqslant (2p)!/2\) for \(p\g 3\), so $\bar{G}$ cannot be an alternating group. 

We can then apply the main result of \cite{M01} and, since we are looking for characters of degree \(p^2\), we can immediately exclude Cases 9--12 and 12--17. Moreover, the previous paragraph eliminates Cases 7, 8 and 13,  so it remains to study Cases 1--6. 

\medskip

\noindent \bf{Case 1.} Checking the orders of the finite simple groups, we see that the only simple group of Lie type of characteristic \(p\g 3\) with \(|G|_p = p^2\) is \(\PSL_2(p^2)\). Since \(|\PSL_2(p^2)|=p^2(p^2-1)(p^2+1)/2\g p^6/4\g p^5\), this case is not possible.

\medskip

\noindent \bf{Case 2.} \(\bar{G}=\PSL_2(q)\) and \(p^2=q\pm 1\), or $q$ is odd and \(p^2=(q\pm 1)/2\). In the first case \(q\g 2\) and \(q\g q-1 \geqslant p^2/2\), so 
\[|\bar{G}| = q(q-1)(q+1)\g p^6/4\g p^5,\]
which is false. In the second case, \(q-1\geqslant (q+1)/2\geqslant p^2\), so
\[|\bar{G}| = q(q-1)(q+1)/2 \g p^6   \]
and this case is not possible.

\medskip

\noindent \bf{Case 3.} \(\bar{G}=\PSL_n(q)\), \(q,n\g 2\) with \((n,q-1)=1\) and \(p^2=(q^n-1)/(q-1)\). Then \(q^n\g p^2(q-1)\g p^2\) and 
\begin{align*}
|\bar{G}|& =q^{n(n-1)/2} \prod_{i=2}^{n} (q^{i}-1) \\
&\geqslant q^n (q^n-1)(q^{n-1}-1)\\
& \g p^4 (q-1)^2 (q^{n-1} - 1)\\
&\g p^4 (q/2)^2 (q^{n-1}/2 )  = p^4q^n/4  \g p^6/4\g p^5
\end{align*}
where we have used that \(q-1\g q/2\g 1\). This is a contradiction. 

\medskip

\noindent \bf{Case 4.} \(\bar{G}=\te{PSU}_n(q)\), \(n\g 2\) and \(p^2=(q^n+1)/(q+1)\). Then \(q^{n-1} \g p^2\) and 
\begin{align*}
|\bar{G}| = \frac{q^{n(n-1)/2}}{(n,q+1)} \prod_{i=1}^n (q^i - (-1)^i) & \geqslant (q^{n-1}/2)(q^n+1)(q^{n-1}-1) \g p^6/2,
\end{align*}
where we have used that \(1/(n,q+1)\geqslant 1/(2q)\). This is a contradiction. 

\medskip

\noindent \bf{Cases 5,6.} \(\bar{G}=\te{PSp}_{2n}(q)\), \(n\g 1\) and \(p^2 = (q^n\pm 1)/2\). Then \(q^n-1\g p^2\) and 
\[ |\bar{G}| \geqslant \frac{q^{n^2}}{(2,q-1)} \prod_{i=1}^n (q^{2i}-1)\geqslant q^n (q^{2n}-1)=q^n(q^n+1)(q^n-1)\g p^6,\]
where we have used that \(q^{n^2}/(2,q-1)\geqslant q^{2n}/(2,q-1)\geqslant q^n\), a contradiction. 
\end{proof}

\medskip

The groups of minimal order with an irreducible character of degree \(pq\), where \(p\) and $q$ are distinct primes, will be treated elsewhere.

\printbibliography

\end{document}